\documentclass[a4paper]{paper}

\usepackage{amsmath,amsfonts,amsthm}
\usepackage{amssymb,latexsym}
\usepackage{fixmath}
\usepackage{mathrsfs,amsbsy}
\usepackage{dsfont}
\usepackage{multirow}
\usepackage{algorithm,algpseudocode}
\usepackage[normalem]{ulem}
\usepackage{natbib}
\usepackage{xcolor}
\usepackage{caption}
\usepackage{subcaption}

\theoremstyle{plain}

\theoremstyle{definition}
\newtheorem{prop}{Proposition}
\newtheorem{theorem}{Theorem}
\newtheorem{remark}{Remark}

\newcommand{\diag}{\mathop{\mathrm{diag}}}
\newcommand{\tr}{\mathop{\mathrm{trace}}}
\newcommand{\argmin}{\mathop{\mathrm{argmin}}}

\usepackage[]{graphicx}

\title{Damping optimization of discrete mechanical systems -- rod/string model}
\author{
  Ninoslav Truhar\thanks{Corresponding author. E-mail: \texttt{ntruhar@mathos.hr},
    Phone: +385912241037.}\\
  \small School of Applied Mathematics and Informatics,\\
  \small Josip Juraj Strossmayer University of Osijek,\\
  \small Trg Ljudevita Gaja 6, 31000 Osijek, Croatia
  \and
  Kre\v{s}imir Veseli\'c\thanks{E-mail: \texttt{kresimir.veselic@fernuni-hagen.de}}\\
  \small Fernuniversit\"at Hagen,\\
  \small Fakult\"at f\"ur Mathematik und Informatik,\\
  \small Postfach 940, D-58084 Hagen, Germany
}

\date{} % You can put a date here or leave it empty

\begin{document}

\maketitle

\begin{abstract}
This paper investigates two optimization criteria for damping optimization in a multi-body oscillator system with arbitrary degrees of freedom ($n$), resembling string/rod free vibrations. The total average energy over all possible initial data and the total average displacement over all possible initial data. Our first result shows that both criteria are equivalent to the trace minimization of the solution of the Lyapunov equation with different right-hand sides. As the second result, we prove that in the case of damping with one damper, for the discrete system, the minimal trace for each criterion can be expressed as a linear or cubic function of the dimension $n$. Consequently, the optimal damping position is determined solely by the number of dominant eigenfrequencies and the optimal viscosity, independent of the dimension $n$, offering efficient damping optimization in discrete systems. The paper concludes with numerical examples illustrating the presented theoretical framework and results.
\end{abstract}

\vspace{2ex}

\bigskip

% Uncomment if you want to provide MSC classification:
% \noindent \textbf{MSC (2010):} 00-xx

%%
%\begin{document}
%
%\keywords{String model, damping optimization, optimal position of a damper, {Lyapunov} equation}
%

\section{Introduction}

One of the crucial problems when attempting to control vibrational systems is the optimal positioning of dampers and the determination of their respective optimal damping values. However, apart from some special cases, the resulting scientific challenges are still open in different disciplines such as e.g. mathematics, engineering, and physics. The problem arises from  considering the wave-like equations, where after an appropriate discretization we consider the dynamics of the second order system of ordinary differential equations, cf. e.\,g.~\cite{BGTQRSS21}.
{The restriction to 1D vibrating systems
(or networks of thereof) is natural,
since meaningful  engineering applications are mostly dealing with only such kind of systems using a finite number of point-dampers.}

In this paper, we will answer the question
of which is the best position for one damper in a given vibrational system (like $n$-mass oscillator in Figure 1).  The best position will be determined by some optimality in the solution of the corresponding system of ODEs, independent of initial data, like initial position and velocity.

% damping” (meaning the type and the damper's positions) which ensures

For the sake of clarity, let us consider the mechanical system described {by} the system of
ordinary differential equations (ODE)
\begin{eqnarray}\label{MnSys}
& M \ddot{\mathbf{x}}(t) + v \cdot D  \dot{\mathbf{x}}(t) + K \mathbf{x}(t) = \mathbf{0} \,,
\\ &  \mathbf{x}_0=\mathbf{x}(0), \mathbf{v}_0= \dot{\mathbf{x}}(0) \,, \nonumber
\end{eqnarray}
where the mass and the stiffness matrices $M$ and $K$ are symmetric positive definite  real matrices of order $n \times n$. The damping is defined as a rank {1} matrix $D  = \mathbf{e}_k \mathbf{e}_k^T $, where $\mathbf{e}_k \in \mathbb{R}^{n}$ is $k$-th canonical basis vector giving the
 dampers' positions and a parameter~$v$ called viscosity. {The vectors $\ddot{\mathbf{x}}$, $\dot{\mathbf{x}}$, and~$\mathbf{x}$ are $\in \mathbb{R}^{n}$ and denote the acceleration, velocity, and displacement, respectively.}

In particular, the ODE system \eqref{MnSys} can be applied on the vibration {chain} of masses and springs shown {on}  Figure \ref{fig1}.

\begin{figure}[htp]
\begin{center}\centering
\includegraphics[scale=0.6]{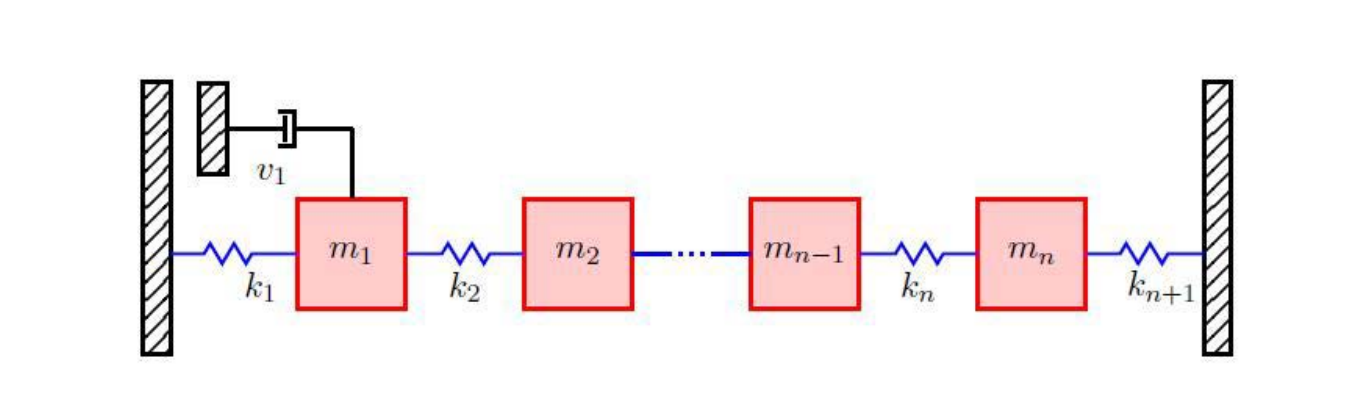}
\caption{The $n$-mass vibration chain with one damper}
\label{fig1}
\end{center}
\end{figure}

Here, the mass matrix is defined as
\begin{equation}\label{matrica M}
M=\diag(m_1,m_2,\ldots,m_n), \quad m_i>0 \,,
\end{equation}
the stiffness matrix is
\begin{equation}\label{matrica K}
K=\left(%
\begin{array}{ccccc}
  k_1+k_2 & -k_2 &  &  &  \\
  -k_2    & k_2+k_3& -k_3 &  &  \\
          & \ddots   & \ddots & \ddots &  \\
          &      & -k_{n-1} & k_{n-1}+k_{n} & -k_n   \\
          &      &  & -k_n & k_n+k_{n+1} \\
\end{array}%
\right).
\end{equation}
In the setting shown on  Figure \ref{fig1}, the damping matrix is
\begin{equation}\label{matrica C}
{v} \cdot D = v_1 \cdot \mathbf{e}_1 \mathbf{e}_1^T, \qquad v_1 > 0 \,,
\end{equation}
with $\mathbf{e}_1 \in \mathbb{R}^n\,, \mathbf{e}_1^T = (1, \, 0, \, \ldots \, , \, 0)$.

Extensive research spanning four decades has focused on the optimal design, placement, and sizing of fluid viscous dampers.
%In \cite{DoRiccTakew19}, a comprehensive review of the predominant methodologies drawn from the vast literature in this domain is provided. The article \cite{DoRiccTakew19} meticulously examines key facets and principal attributes of various approaches aimed at determining the optimal damping coefficients and the most suitable locations for fluid viscous dampers, offering a comparative analysis.
An efficient and systematic procedure for finding the optimal damper positioning to minimize the amplitude of a transfer function of a cantilever beam has been considered in \cite{Takewaki98,Aydin14,SanAvilBer18}.

For the problem of finding ``best damping'', one of the crucial choices is the optimization criterion, which is a choice of a penalty function that has to be minimized. Once the optimization criterion has been defined (or determined) in most (existing) cases one runs through ``all possible positions'' and then the best position is one with the smallest penalty function. Our goal will be to find a way how efficiently one can calculate the best position of {an} (additional) external damper for two different criteria (penalty function).

The first optimization criterion will be
\begin{itemize}
  \item[i)] \textbf{The total average energy over all possible initial data.}
\end{itemize}
 The second criterion will be
 \begin{itemize}
  \item[ii)] \textbf{The total average displacement over all possible initial data.}
\end{itemize}

The problem of finding the optimal positioning of a viscous damper for a linear conservative mechanical system based on an energy criterion has been studied in many papers such as~\cite{GurgozeMuller92,MuellerWeber21,Takewaki98,VES89,VES90},
  or lately in {the} last couple of decades, in~\cite{DymarekDzitkowski21,LiuRaoZhang20, MorzfKawMa2013, TRUHVES05,SilvestriTrombetti06,Cox:04,TRUH04,TRUHVES09}.

On the other hand, the total average displacement over all possible initial data is a novel concept to the best of our knowledge. Usually, the average displacement (amplitude) optimization is connected with problems treated in e.\,g.~\cite{ITakewaki2009,YKanno2013}. Recently, an overview of optimal damper placement methods in different structures has been given in \cite{KookShenZhuLind21}.

Comprehensively study  the controllability and
stability of second order infinite dimensional systems coming from elasticity can be found in \cite{AN15}. Besides many results like stabilization of second order evolution equations by a class of unbounded feedbacks, stabilization of second order evolution equations
with unbounded feedback with delay or systems without delay in \cite{AN15}
as well as in \cite{AmmHerTucs01} authors present results on the optimal location of one damper which calms down the whole undamped spectrum of the corresponding pair $(M,K)$. The optimization criterion used in \cite{AN15} or in \cite{AmmHerTucs01} is similar to the minimization of the total average energy (criterion i)).

To enhance comprehension regarding the main motivation  behind this paper we will briefly explain the first criterion i) \emph{the mean value of the total energy}, which will be explained in more detail in section \ref{subsubsec_criterion_total_energy}.
The minimization of the mean value of the total energy  is equivalent to the minimization of the trace
\begin{equation*}% \label{Trace_0}
\tr( Z X(p,v) ) \rightarrow \min \,,
\end{equation*}
where $X$ is a solution of the {Lyapunov} equation
\begin{equation*}% \label{Lyap eq 1}
A^T X(p,v) + X(p,v) A = -I_{2n},
\end{equation*}
where $A$ is defined using $M$, $D$ and $K$ (see below equation~\eqref{OptA}), and $v>0$ and $p$ are damper's viscosity and its position, respectively.
Let \(\omega_1, \ldots, \omega_n\) be the undamped eigenfrequencies, that is, the square roots of the eigenvalues of the matrix pair \((K, M)\).

The matrix $Z= Z_s \oplus Z_s$ depends on the part of the spectrum that we try to calm.  For example, if we are interested in the best way to calm the first $s$ undamped eigenfrequencies, that is  $0 < \omega_1 < \omega_2 < \ldots < \omega_{s}$ of the undamped system, the matrix $Z$ will have the following form
\begin{equation*}
Z= Z_s \oplus Z_s\, \qquad Z_s =\begin{bmatrix} I_{s} &  \\  & 0_{(n-s)}  \end{bmatrix}\,.
\end{equation*}

%Although, the main ideas which we will present in this
%paper hold for general $M$ and $K$, the proofs are available only for the models whose eigensystem can be expressed with closed formulas.
%

Although numerical experiments suggest that the main ideas presented in this paper hold for general
$M$ and $K$, at present we can only provide rigorous proofs for models whose eigensystems admit closed-form expressions.

Through this paper we consider the mechanical system  \eqref{MnSys} where all $m_1 = \ldots = m_n=1$, $k_1 = \ldots = k_{n+1}=1$ in \eqref{matrica M}
and \eqref{matrica K}, respectively. This means that
\begin{equation}\label{string_discret}
M=I_n \,, \qquad K=\left(%
\begin{array}{ccccc}
  2 & -1 &  &  &  \\
  -1    & 2& -1  &  &  \\
          & \ddots   & \ddots & \ddots &  \\
          &      & -1 & 2 & -1   \\
          &      &  & -1 & 2 \\
\end{array}%
\right)\,,
\end{equation}
which corresponds with the discretization of the system of ordinary differential equations \eqref{MnSys}.

Until now, the optimization of damper positions has primarily been approached using heuristic methods, without rigorous proofs.
Examples of such approaches can be found in \cite{BenTomTruh2011}, \cite{BenTomTruh2013} , \cite{BenKurTomTruh2016}.

Numerical experiments indicate that the main ideas presented in this paper hold for a general definite pair $(M, K)$, where
$M$ and $K$ can be simultaneously diagonalized with a positive spectrum. However, at the moment, for the discrete systems,
we can only prove results for $M$ and $K$ as in \eqref{string_discret}.

The main result of this paper shows that, for discrete system, both criteria i) total average energy and ii) the total average displacement,
the optimal position $\textbf{p}_{\textrm{opt}}=p_{\textrm{opt}}/n$,  where $1 \leq p_{\textrm{opt}} \leq n$ depends only on the number of dominant eigenfrequencies $s$ and does not depend on dimension $n$.

To highlight the difference  between the damping of the whole and the
part of the spectrum for the string/rod vibrations, let us consider the system of ODE \eqref{MnSys} of dimension $n=600$. Further, let $M$ and $K$ be as in \eqref{string_discret}.

As the  first case we consider the problem of calming down the whole spectrum (all undamped eigenfrequencies $\omega_1 < \ldots < \omega_{n}$), that is the case when $Z= I_{2 n}$.  The first illustration presents all $n=600$  minimal traces, $\tr( X(v_\textrm{opt},p) )$, $v_\textrm{opt}$ is optimal viscosity  and  $p$ is position $1 \leq p \leq n$ and \textbf{p}=$p/n \in\{\frac{1}{n}, \frac{2}{n}, \ldots, 1\}$ (see Figure \ref{fig1A}).

\begin{figure}[htp]
\begin{center}\centering
\includegraphics[scale=0.324]{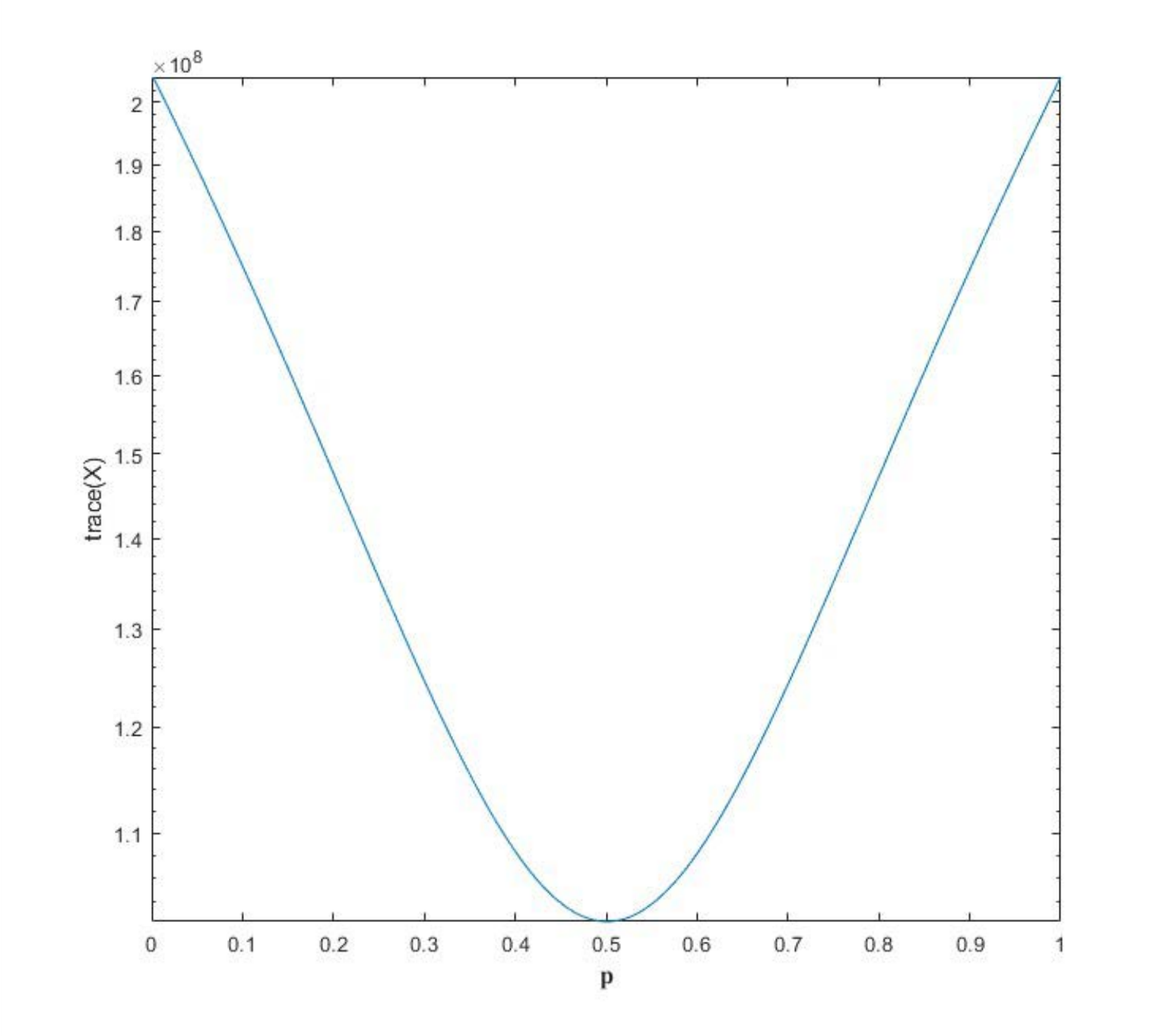}\caption{The optimal trace as a function of positions; damping of the whole spectrum.}
\label{fig1A}
\end{center}
\end{figure}

As one can see from  Figure \ref{fig1A} the optimal position
is located at the middle of the string, consistent with the result presented in \cite{AN15} or in \cite{AmmHerTucs01}.

Although, the main result from \cite{AmmHerTucs01} as well as  Figure \ref{fig1A} shows that the optimal location of just one damper which calms the whole spectrum uniformly has been solved and thus is not a challenge anymore we will show that the problem of calming down part of the spectrum
is completely different and still far away from a general solution.

Thus, the second Figure \ref{fig1B} shows all $n=600$ minimal $\tr( X(v_\textrm{opt},p) )$, $v_\textrm{opt}$ is optimal viscosity and $p$ is position $1 \leq p \leq n$ for the problem when one tries to calm the first 20 undamped eigenfrequencies $\omega_1 < \ldots < \omega_{20}$. This means that we have chosen a model in which the undamped eigenfrequencies $\omega_1 < \ldots < \omega_{20}$ are, in a certain sense, dominant. A similar problem has been considered in~\cite{TRUH04,TRUHVES05,TRUHVES09}. Here $Z= Z_s \oplus Z_s$, and $Z_s$ has rank $s=20$.

\begin{figure}[htp]
\begin{center}\centering
\includegraphics[scale=0.33]{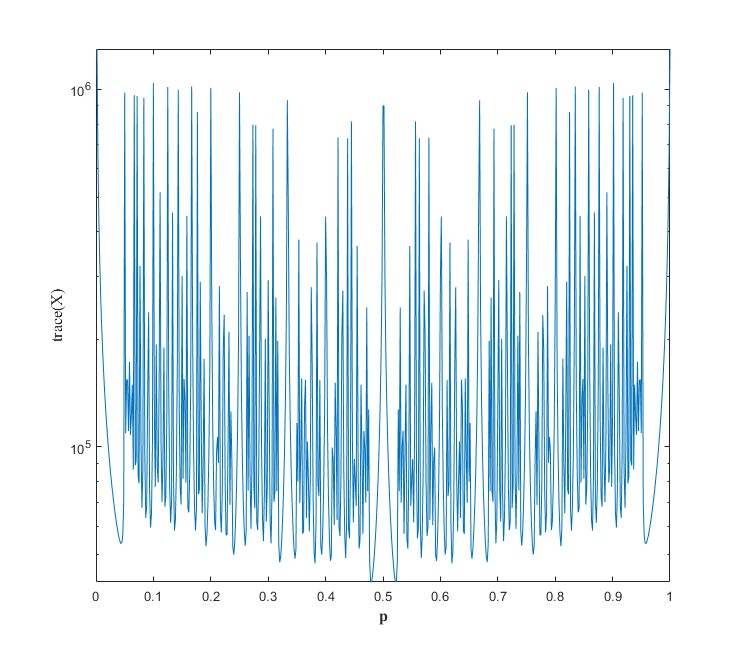}
%\scalebox{1}{\input{StringN600Omega1to20.jpg}}
\caption{The optimal trace as a function of positions;  damping of the part of the spectrum.}
\label{fig1B}
\end{center}
\end{figure}

As one can see from Figure \ref{fig1B} the optimal location of the damper is at $0.48$ (or due to the symmetry at $0.52$)  with several different local minima.

Moreover, if one is interested in calming down some small part of the undamped spectrum $\omega_{i+1} < \ldots < \omega_{i+s}$, with $s \ll n$ we will show that the optimal position depends on $s$ as well as $i$.

In that sense Figure \ref{fig1C} shows the optimal traces when one tries to calm down spectrum $\omega_{51} < \ldots <\omega_{70}$.
The best position is at $0.0083$.

\begin{figure}[htp]
\begin{center}\centering
\includegraphics[scale=0.5]{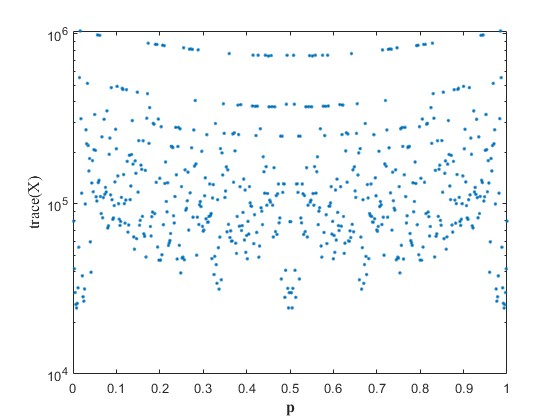}
\caption{The optimal trace as a function of positions;  damping of the part of the spectrum.}
\label{fig1C}
\end{center}
\end{figure}

The above illustrations show that the calculation of the optimal locations together with the corresponding viscosity may be a very demanding task.

In this paper, we will present several new results connected with the
calculation of the optimal location together with the corresponding viscosity
for one damper which calms down a part of the spectrum of
string or rod model described by ODE system \eqref{MnSys} discretized such that \eqref{string_discret} holds.

The first result is a new approach to defining the \emph{total average displacement criterion over all possible initial data}, criterion ii).
Based on the formula for the trace of the solution of Lyapunov equation
from \textsc{Veseli\'{c}} \cite{VES90}  or \cite{BRAB98} we develop
a new very efficient numerical procedure for calculating the optimal viscosity for the criterion of the total average displacement over all possible initial data.

The second result, which we have already mentioned,  claims that the optimal position of the damper depends only on the number of dominant eigenfrequencies $s$  and their location $i$ ($\omega_{i+1} < \ldots < \omega_{i+s}$)  and does not depend on dimension $n$. It holds for both criteria.

The paper is organized as follows. In Section \ref{subsec_opt_criteria}
we present two optimization criteria. The \emph{average total energy} criterion is briefly presented in Section \ref{subsubsec_criterion_total_energy}. In Section \ref{subsubsec_criterion_total_disp} we present a novel approach in more detail for the \emph{average total displacements} criterion.
Section \ref{sec_optimal_positions} contains a precise analysis of the
string or rod free vibrations, with exact expressions for undamped
eigenfrequencies as well as corresponding eigenvectors.
Sections \ref{sec_optimal_damp_positions} and \ref{Opt pos_single_damper2}
contain some auxiliary results and notations. The main results one can find in Sections \ref{Opt_trace_aver_Energy} and \ref{Opt_trace_aver_Displac}, where we have shown that the minimal trace for the \emph{average total energy} criterion is a linear function and for the \emph{average total displacements} criterion is a cubic function of dimension $n$, respectively. Finally in  Section  \ref{sec_numerical_example} we present several numerical examples which illustrate results from previous sections.

Throughout the paper, $M$ denotes the mass matrix, $K$ represents the stiffness matrix, and $c$ is a vector used to define the corresponding damping matrix $C = v c c^T $, where $ v > 0 $ is a real parameter. The values of $ M $, $ K $, and $c$ may vary depending on the specific model (e.g., string, road, or discrete mass--spring system).

\section{Optimization criteria\label{subsec_opt_criteria}}

Although the system depicted in Figure \ref{fig1} assumes that the mass matrix is diagonal, our approach is more general and we can treat any definite pair~$(M, K)$.

In fact, in our approach, as a preprocessing step, an eigenvalue decomposition of the pair $(M, K)$  is done. Since $M$ and $K$ are symmetric positive definite there exists a nonsingular matrix $\Phi$ such that
\begin{equation}\label{simult.diagofMK}
\Phi^T K \Phi = \Omega^2= \diag (\omega_1^2, \ldots, \omega_n^2 )
\quad \text{and} \quad \Phi^T M \Phi = I_n\,,
\end{equation}
$0 < \omega_1 < \omega_2 < \ldots < \omega_n$ ($\omega_i$ are the eigenfrequencies of the undamped system). Now, multiplying equation~\eqref{MnSys} from the left-hand side with $\Phi^T$ and using equation~\eqref{simult.diagofMK} we get
\begin{eqnarray}\label{MnSys-PhiT}
\Phi^{-1} \ddot{\mathbf{x}}(t) + v \cdot \Phi^T D \Phi \Phi^{-1}  \dot{\mathbf{x}}(t) + \Omega^2 \Phi^{-1} \mathbf{x}(t) = \mathbf{0} \,,
\end{eqnarray}
where $D = v \cdot \mathbf{e}_k \mathbf{e}_k^T $, represents a damper at $k$-th position ($k\in \{1, 2, \ldots, n\}$). By introducing the substitution
\begin{eqnarray}\label{FirstOrdODE}
\mathbf{y}_1(t) =  \Omega \Phi^{-1} \mathbf{x}(t) \,, \quad \mathbf{y}_2(t) = \Phi^{-1} \dot{\mathbf{x}}(t) \,,
\end{eqnarray}
we get the corresponding first order ODE,
\begin{eqnarray}
\frac{\mathrm{d}}{\mathrm{d} t}\begin{bmatrix} \mathbf{y}_1(t) \\ \mathbf{y}_2(t) \end{bmatrix} =
  \begin{bmatrix} 0 & \Omega \\ -\Omega & -C   \end{bmatrix} \begin{bmatrix} \mathbf{y}_1(t) \\ \mathbf{y}_2(t) \end{bmatrix} =: A\begin{bmatrix} \mathbf{y}_1(t) \\ \mathbf{y}_2(t) \end{bmatrix}
\,,\label{OptA}
\end{eqnarray}
where the damping matrix is
\begin{eqnarray}\label{Def:dampC}
C =  v \cdot \Phi^T D \Phi \,.
\end{eqnarray}

\subsection{Optimization criterion -- average total energy\label{subsubsec_criterion_total_energy}}

Note that, if we write
\begin{eqnarray*}
\mathbf{y} = \begin{bmatrix} \mathbf{y}_1(t) &  \mathbf{y}_2(t) \end{bmatrix}^T\,,
\end{eqnarray*}
from equation~\eqref{FirstOrdODE} follows
\begin{equation*}
   \mathbf{y}(t)^T \mathbf{y}(t) =
{\| \mathbf{y}_1(t) \|}^2 + {\| \mathbf{y}_2(t) \|}^2 =  \mathbf{x}^T K \mathbf{x} + \dot{\mathbf{x}}^T M \dot{\mathbf{x}} = 2 E(t) \,.
\end{equation*}
In other words, the \textsc{Euclidian} norm of this phase-space representation
equals twice the total energy of the system. From this, it follows that
{\em all phase space matrices are unitarily equivalent}. Thus, for all
total-energy relevant considerations, we may choose any of these
representations at our convenience.

For the first optimization criterion, we will use a minimization of the mean value of the total energy.
In \cite{NAKIC02,VesBrabDel01} it has been shown that this optimization criterion is equivalent with
the minimization of the trace
\begin{equation}\label{Trace_0}
\tr( Z X ) \rightarrow \min \,,
\end{equation}
where $X$ is a solution of the {Lyapunov} equation
\begin{equation}\label{Lyap eq 1}
A^T X + X A = -I,
\end{equation}
where $A$ is defined in equation~\eqref{OptA}
and $Z= Z_s \oplus Z_s$ depends on the part of the spectrum that we try to calm.

For example, if we are interested in the best way to calm the first $s$ eigenfrequencies
$0 < \omega_1 < \omega_2 < \ldots < \omega_{s}$ of the undamped system, the matrix $Z$ will have the following form
\begin{equation}\label{def_Z}
Z= Z_s \oplus Z_s\, \qquad Z_s =\begin{bmatrix} I_{s} &  \\  & 0_{(n-s)}  \end{bmatrix}\,.
\end{equation}
For more details about the construction of the matrix $Z$, see, for example, \cite{NAKIC02}.

It is easy to show that the trace minimization \eqref{Trace_0} with {Lyapunov} equation
\eqref{Lyap eq 1} is equivalent to the trace minimization of the solution of the so-called dual
{Lyapunov} equation of the form
\begin{equation}\label{Lyap eq Dual}
A Y + Y A^T = -Z,
\end{equation}
since
\begin{equation}\label{DualXandY}
\tr( Y ) = \tr( Z X )\,.
\end{equation}

%%%%%%%%%%%%%%%%   NEW %%%%%%%%%%%%%%

\subsection{Optimization criterion -- average total displacements\label{subsubsec_criterion_total_disp}}

Since, to the best of our knowledge  similar theory for the average total displacements as the above one (for the average total energy) has not been presented yet, in this section we will present the basic results on {Lyapunov} theory for lowering displacements, which will be our second criterion.
%
%\begin{equation}\label{1-MCK}
%M \ddot{\mathbf{x}} + C \dot{\mathbf{x}} + K \mathbf{x} = 0.
%\end{equation}

Recall that in equation~\eqref{FirstOrdODE} we have defined
\begin{equation*}
\mathbf{y}_1(t) =  \Omega \Phi^{-1} \mathbf{x}(t) \,, \quad \mathbf{y}_2(t) = \Phi^{-1} \dot{\mathbf{x}}(t) \,,
\end{equation*}
which together with equation~\eqref{OptA} gives
\begin{equation*}
\frac{\mathrm{d} \mathbf{y} }{\mathrm{d} t} = A \mathbf{y} \,,
  \end{equation*}
with
\begin{equation}\label{A}
A=\begin{bmatrix} \mathbf{0} & \Omega \\ -\Omega & -C   \end{bmatrix}\,,
\end{equation}
which is solved by
\begin{equation}\label{22}
\mathbf{y}= e^{At}
\left[%
\begin{array}{c}
  \mathbf{y}_{10} \\
  \mathbf{y}_{20} \\
\end{array}
\right]   \,,
 \end{equation}
where $\mathbf{y}_{0}=\begin{bmatrix} \mathbf{y}_{10} & \mathbf{y}_{20} \end{bmatrix}^T$ contains {the} initial data.

Let
\begin{align*}
  K = U_K \Lambda_K U_K^T \,,
\end{align*}
be the eigenvalue decomposition of the matrix $K$. Note that
\begin{align*}
   \Omega^{-1} \Phi^{T} U_K \Lambda_K U_K^T \Phi \Omega^{-1} = I_n\,,
\end{align*}
which means that $\Omega^{-1} \Phi^{T} U_K$ is unitary similar to the diagonal matrix
$\Lambda_K^{-1/2}$, which further implies that we can write the singular value decomposition
\begin{align*}
   U_K^T \Phi \Omega^{-1} = U \Lambda_K^{-1/2} V^T \,,
\end{align*}
or
\begin{align} \label{defhatK}
   \Omega^{-1} \Phi^{T} \Phi \Omega^{-1}  = V \Lambda_K^{-1} V^T \doteq \widehat{K}^{-1} \,,
\end{align}
which means that $\widehat{K}$ is unitary similar to $K$, here $U$ and $V$ are orthogonal matrices.

The quantity to be minimized is the mean displacement \(\hat{\mathbf{x}}\) given as
\begin{align}\label{1-hatx}
\hat{\mathbf{x}}^2 = &  \int_0^\infty\|\mathbf{x}(t)\|^2\mathrm{d}t = \int_0^\infty \mathbf{y}_1(t)^T \Omega^{-1} \Phi^{T} \Phi \Omega^{-1} \mathbf{y}_1(t)\mathrm{d}t  \nonumber \\
& = \int_0^\infty\left(P_1e^{At}\mathbf{y}_0\right)^T \widehat{K}^{-1}P_1e^{At}\mathbf{y}_0 \mathrm{d}t,
 \end{align}
where $\widehat{K}$ is defined in equation~\eqref{defhatK} and
\[
P_1 =
\left[%
\begin{array}{cc}
I_n & \mathbf{0} \\
\end{array}
\right].
\]
So,
\begin{eqnarray}\label{2-hatx}
\int_0^\infty\|\mathbf{x}(t)\|^2\mathrm{d}t  =
\mathbf{y}_0^T\hat{X}\mathbf{y}_0, & \\
\hat{X} = \int_0^\infty e^{A^Tt}Ze^{At}\mathrm{d}t, &
Z = P_1^T\widehat{K}^{-1}P_1 =
\left[\begin{array}{cc}
\widehat{K}^{-1} & 0 \\
0      & 0 \\
\end{array}
\right] , \label{Zfordisplacmt}
\end{eqnarray}
where \(\hat{X}\) solves the {Lyapunov} equation
\begin{equation}\label{lyaphat}
A^T\hat{X} + \hat{X} A = -Z\,.
 \end{equation}
\begin{prop} The matrix \(\hat{X}\) is symmetric positive definite.
\end{prop}
\begin{proof} \(Z\) is positive semidefinite,  so is \(\hat{X}\) as well.
Assume \(\hat{X}z = 0\), then by equation~(\ref{2-hatx})
and the positive definiteness of $\widehat{K}$
\[
\mathbf{x}(t) = \widehat{K}^{-1/2}P_1e^{At}z = 0,\quad \forall t \geq 0.
\]
hence also \(\dot{\mathbf{x}}(t) = \mathbf{0}\) for all \(t \geq 0\) {and thus} \(z = 0\).

\end{proof}

The unit sphere average
\[
\tr(\hat{X}) =  \int_0^\infty \tr \left(e^{At}e^{A^Tt}Z\right) \mathrm{d}t = \tr \left(JXJZ\right) =
\tr \left(XZ\right)
\]
where the matrix \(J\) is defined as
\[
J=\begin{bmatrix} I_n & \mathbf{0} \\ \mathbf{0} & -I_n \end{bmatrix} \,,
\]
and \(X\) solves the standard {Lyapunov}
\[
A^TX + X A = -I \,.
\]
We have used the \(J\)-symmetry of \(A\) and the fact that \(J\)
and \(Z\) commute.

Frequency-cut average with a frequency-cut projection \(P\) can be obtained as
\begin{align*}
\tr(\hat{X}P) & =  \tr \int_0^\infty \left(e^{A^Tt}Ze^{At}P\right)\mathrm{d}t  \\
& = \tr \left(Z\int_0^\infty e^{At}Pe^{A^Tt} \mathrm{d}t \right) = \tr (Z \hat{Y}) \,,
 \end{align*}
where
\[
A\hat{Y} + \hat{Y} A^T = -P \,,
\]
and since \(P,\ J\) commute and also \(Z,\ J\)
\[
\tr(\hat{X}P) = \tr (Z \tilde{Y})
\]
where \( \tilde{Y}\) solves the standard {Lyapunov} equation
\[
A^T\tilde{Y} + \tilde{Y} A = -P \,.
\]

Note, that in both criteria one needs to minimize a trace of the corresponding {Lyapunov} equation,
with the same system matrix but a different right-hand side.

As an illustration that the optimal viscosities are different for each criterion, we present a small example.

\begin{example}

Consider the one dimensional oscillator
\[
m\ddot{x} + c\dot{x} + kx = 0\,,
\]
and the {Lyapunov} equation
\[
\left[\begin{array}{cc}
0 & -\sqrt{k/m} \\
\sqrt{k/m}  & -c/m \\
\end{array}
\right]
\left[\begin{array}{cc}
x_{11} & x_{12} \\
x_{12}   & x_{22} \\
\end{array}
\right]
+
\left[\begin{array}{cc}
x_{11} & x_{12} \\
x_{12}   & x_{22} \\
\end{array}
\right]
\left[\begin{array}{cc}
0 & \sqrt{k/m} \\
-\sqrt{k/m}  & -c/m \\
\end{array}
\right]
\]
\[
=
\left[\begin{array}{cc}
-a & 0 \\
0    & -b \\
\end{array}
\right]
\]
with \(a,b \geq 0,\ a+ b > 0\).
This gives the equations
\[
-2 \sqrt{k/m} x_{12} = -a\qquad
\sqrt{k/m}(x_{11} - x_{22}) - x_{12}c/m = 0
\]
\[
x_{12}\sqrt{k/m} - x_{22}c/m + x_{12}\sqrt{k/m} - x_{22}c/m = -b\,.
\]
Hence
\[
x_{12} = \frac{a}{2\sqrt{k/m}} =  \frac{a}{2}\sqrt{m/k} \,,
\]
\[
(x_{22} - x_{11})\sqrt{k/m} +  \frac{ac}{2\sqrt{km}} = 0\,,
\]
\[
x_{22} = \left(\frac{a}{2} + \frac{b}{2}\right)m/c \,,
\]
\[
x_{11} =   x_{22} +  \frac{ac}{2k} = \frac{a+b}{2}m/c + \frac{ac}{2k}.
\]
 For \(a = b = 1\) we obtain the trace for the average energy criterion
\[
\tr(X_1) = 2m/c + \frac{c}{2 k}\,,
\]
with the known minimum at the critical damping \(c = 2\sqrt{mk}\).
In the case of the average displacement criterion we have \(a = 1/k,\ b = 0\) giving the trace
\[
\tr(X_2) = m/(kc) + c/(2k^2) \,,
\]
again with a unique minimum  \(c = \sqrt{2} \sqrt{mk}\) obviously different from the above one.
\end{example}

This shows that further study of the properties of both criteria
will be an interesting issue, especially a comparison between them in the sense of quality of the solution and complexity of the calculation. However, in this paper, we will not further consider the optimization criteria, this will be a subject of our future studies.
In the rest of the paper, we will concentrate only on the position
optimization for both criteria for the vibrations of the structured models.

\section{Vibrations of rod/string and discrete mass--spring model}
\label{sec_optimal_positions}

% \section{Vibrations of a rod/string}

Our main result concerns string or rod vibrations described with the following partial differential equation:
\begin{align}\label{string}
 \rho(x) u(x,t)_{tt} + c(x) u(x,t)_t - (k(x)u(x,t)_x)_x = 0\,, & u(0,t)=u(1,t)=0 \,, \quad
  x \in (0,1)\,,
\end{align}
with point damping $c(x) = \delta(x-y)$ concentrated in the position $x = y$.

The advantage of the equation \eqref{string} is in the fact that it depends continuously on $y$ allowing analytical minimization. Transport this
to the discretization. The damper position $y$ must stay continuous. Take $\rho(x) = 1$, $k(x) = 1$ for simplicity. Using the method of separating variables or the Fourier method one gets the undamped system
\begin{align*}
 \omega^2 u + u(x,t)_{xx} = 0\,,
\end{align*}
which is explicitly diagonalized with the undamped frequencies $\omega = j \pi$  and the orthonormal eigenvectors
\begin{align*}
 u_{k}(x) = \frac{1}{\sqrt{2}}  \sin{k \pi x} \,.
\end{align*}
In the weak formulation of the formal equation \eqref{string} the delta damper is represented by the quadratic form
\begin{align*}
 \omega(u,v)  = d u(y) v(y) \,, d \geq 0 \,,
\end{align*}
to which there corresponds the matrix $C(y)$ in the basis $u_k$, given as
\begin{align*}
 C_{kj} = \frac{d}{2}  \sin{(k \pi y)}  \sin{(j \pi x)}\,.
\end{align*}
Now a discretization consists in cutting a finite matrix out of these.

This leads to the finite dimensional system with
\begin{align*}
K = \diag(\omega_1^2, \ldots, \omega_n^2) \,, M=I_n\,, C=C(y) = d c(y) c(y)^T\,,
\end{align*}
with
\begin{align}\label{cody}
c(y) = \frac{1}{\sqrt{2}} \begin{bmatrix} \sin{(\pi y)} & \sin{(2 \pi y)} & \ldots & \sin{(n \pi y)} \end{bmatrix}^T \,.
\end{align}
More details about the properties of the vibrating system, described by the  partial differential equation \eqref{string} can be found in \cite[4.1 The Wave Equation]{AN15}.

If $\rho$ and $k$ are not constant we may still use the Fourier decomposition above and by common integrals compute the coefficients of the matrices $M$ and $K$ as follows
\begin{align*} % \label{massmatrix}
M_{kj} = \frac{1}{2} \int\limits_0^1 \sin{(k \pi x \rho(x))}  \sin{(j \pi x)} dx \,, \\
K_{kj} = \frac{k j \pi^2}{2} \int\limits_0^1 \cos{(k \pi x k(x))}  \cos{(j \pi x)} dx \,,
% \label{stiffmatrix}
\end{align*}
and $C$ as above. The only difference is that now the matrices $M$ and $K$ are not diagonal anymore.
Note that choosing other boundary conditions would lead to other trigonometric functions.
Thus, we end up with minimizing the Lyapunov trace for one-dimensional damping as described in \cite{VES90} which is given by an explicit formula as a function of $y$. This can again be
minimised either analytically or numerically or just by plotting.

To do this we must first normalise $c(y)$ to a constant norm that is independent of $y$.
We have
\begin{align}
\|c(y)\|^2 & =  \frac{1}{2} \sum\limits_{k=1}^n
\sin^2{(k \pi y)} = \frac{n}{4} - \frac{1}{4}
\sum\limits_{k=1}^n \cos{(2 k \pi y)} = \nonumber \\
& \frac{n}{4} - \frac{1}{4} \emph{Re}
\left( \sum\limits_{k=1}^n e^{2 k \pi \imath y} \right) = \nonumber  \\
& \frac{n}{4} - \frac{1}{4} \emph{Re}
\frac{ e^{2 \pi \imath y} (1 -  e^{2 \pi \imath n y}) }{1 -  e^{2 \pi \imath n y}} \,, 0 < y < 1 \,.
\label{normc}
\end{align}

The same approach can be applied to the rod  model
\begin{align}\label{rod/string}
 \rho(x) u(x,t)_{tt} + c(x) u(x,t)_t - (k(x)u(x,t)_x)_x + (a(x) u(x,t)_{xx})_{xx}  = 0\,, \\ u(0,t)=u(x,t)_{xx}(0,t)=u(1,t)=u(x,t)_{xx}(1,t)=0 \,. \nonumber
\end{align}
If the material is homogeneous, say, $\sigma(x) = \sigma_0 > 0$, $\rho(x) = 1$, $k(x) = k_0$ and $a(x) = a_0$ then the stationary undamped equation
\begin{align} \label{rodstr_stat}
-\omega^2 u - k_0  u(x,t)_{xx} + a_0 u(x,t)_{xxxx} = 0
\end{align}
has again the eigenfunctions
\begin{align*}
u_k(x) =  \frac{1}{\sqrt{2}} \sin{k \pi x} \,,
\end{align*}
with the eigenfrequencies
\begin{align*}
\omega_k = k \pi \sqrt{k^2\pi^2 a_0 +k_0} \,.
%\omega_k = k \pi \sqrt{k_0 + a_0 k^2} \,.
\end{align*}
This was the spectral discretisation of the continuous system. Another way is to start from the finite dimensional case as in \eqref{matrica M} and
\eqref{matrica K}, that is
\begin{equation*}
M=\diag(m_1,m_2,\ldots,m_n), \quad m_i>0 \,,
\end{equation*}
\begin{equation*}
K=\left(%
\begin{array}{ccccc}
  k_1+k_2 & -k_2 &  &  &  \\
  -k_2    & k_2+k_3& -k_3 &  &  \\
          & \ddots   & \ddots & \ddots &  \\
          &      & -k_{n-1} & k_{n-1}+k_{n} & -k_n   \\
          &      &  & -k_n & k_n+k_{n+1} \\
\end{array}%
\right).
\end{equation*}
For $k_i = m_i = 1$ the eigenfrequencies are
\begin{align*}
\omega_l =   2 \sin\frac{l \pi}{2 (n+1)}\,,  l= 1, \ldots, n \,.
\end{align*}
and the eigenvector matrix $Q = (q_{rl})$ given as
\begin{align*}
q_{rl} =  \sqrt{\frac{2}{n+1}} \sin \frac{r l \pi}{n+1} \,.
\end{align*}
That is, the $k$-th eigenvector is given by
\begin{align} \label{eigvecMK}
u_k = \sqrt{\frac{2}{n+1}} \begin{bmatrix} \sin{\frac{1 k \pi }{n+1}} & \sin{\frac{2 k \pi }{n+1}}& \ldots & \sin{\frac{n k \pi }{n+1}} \end{bmatrix}^T
\end{align}
This case is interesting in itself and also because it can be understood as a discretisation of the continuous string. Recall that we have denoted by $Q = (q_{rl})$ the eigenvector matrix. This matrix happens also to be symmetric, so its $k$-th column, that is, the vector $u_k$ from  \eqref{eigvecMK} is at the same time the representation of the canonical unit vector $e_k$ in the orthonormal basis $u_1, \ldots, u_n$. Our
damped system with the damping matrix $C = v e_k e_k^T$ is now equivalent with the system
\begin{align*}
M=I_n\,, K = \diag(\omega_1^2, \ldots, \omega_n^2) \,, C = v u_k u_k^T \,,
\end{align*}
which describes the damper at the position $k$. We now make the discrete variable $k$ continuous by setting
\begin{align*}
C(z) = v c(z) c(z)^T \,,
\end{align*}
with
\begin{align} \label{eigvecMK-cont}
c(z) = \sqrt{\frac{2}{n+1}} \begin{bmatrix} \sin{\frac{1 z n \pi }{n+1}} & \sin{\frac{2 z n \pi }{n+1}}& \ldots & \sin{\frac{n z n \pi }{n+1}} \end{bmatrix}^T \,, 0 < z \leq 1 \,,
\end{align}
such that
\begin{align*}
c\left(\frac{k}{n} \right) = u_k \,.
\end{align*}

Now, this is strikingly analogous to the spectral discretisation we made previously. Moreover, for large $n$ the vectors $c(z)$ here and there coincide. The normalization formulae here are quite analogous to those in \eqref{normc}. Here too, this construction has started from the special stiffness matrix $K$ above but it can serve with any 1D vibrational system where the boundary conditions should better be
taken accordingly.

\subsection{Optimal dampers' position -- discrete mass spring mechanical system}
\label{sec_optimal_damp_positions}

In this section, we will present the main result of the optimal position of one damper for the  discrete mass spring mechanical system.

We will show, that optimal position is a simple function of the dimension $n$, which means that we can calculate the optimal position for one damper on some smaller dimensions and then estimate it for a general case.

We focus on a single damper due to the fact that only for
one damper there exists an explicit expression for the trace of the solution to the corresponding Lyapunov equation.

The initial criterion to be applied is the total average energy criterion. As outlined in section \ref{subsec_opt_criteria}, we will consider the Lyapunov equation
\begin{align*}
A X + X A^T = -Z \,, % &  \qquad A=\begin{bmatrix} 0 & \Omega \\ -\Omega & -C   \end{bmatrix}\,,
  \end{align*}
where $A$ is defined in \eqref{A} and
\begin{align} \label{DefZ2b}
Z_1= Z_s \oplus Z_s\, \qquad Z_s =\begin{bmatrix} I_{s} &  \\  & 0_{(n-s)}  \end{bmatrix}\,.
 \end{align}
Our goal is to find the minimal $tr(X)$, which corresponds with the minimization of the total average energy.
Further,
\begin{align*}
 \Omega= \diag{(\omega_1, \ldots, \omega_n)}\,,
 \omega_j = 2 \sin{\frac{j \pi}{2 (n+1)}}\,.
  \end{align*}
For a given position $p_k $, the damping matrix $C$ is defined as
\begin{align} \label{def_damping_string}
 C = c c^T \,,
  \end{align}
where for $p_k  \in\{\frac{1}{n}, \frac{2}{n}, \ldots, 1\}$
\begin{align*}
c(p_k ,n) = \sqrt{\frac{2}{n+1}} \begin{bmatrix} \sin{\frac{1 p_k  n \pi }{n+1}} & \sin{\frac{2 p_k  n \pi }{n+1}}& \ldots & \sin{\frac{n p_k  n \pi }{n+1}} \end{bmatrix}^T \,.
\end{align*}
Note, that $p_k=k/n$ is a discrete set of points defined by continuous variable $0< z \leq 1$ from \eqref{eigvecMK-cont}.

Since,
\begin{align*}
\omega_j= \omega_j(n)\,, % c_j = c_j(p_k , n)\,,
\end{align*}
 we see that the $\tr(X)$ is a function of dimension $n$ positions $p_k $, and viscosity $v$, that is
\begin{align*}
\tr(X) = \tr(X(p_k , v, n))\,.
\end{align*}

We will show that for a fixed position $\mathbf{p}$ (which does not depend on $n$) defined as
\begin{align} \label{DefposGeneral}
\mathbf{p} \in (0, 1) \,,
\end{align}
the optimal  $\tr(X(\mathbf{p}_{\rm {opt}}, v_{\rm {opt}}))$
is a simple function (linear for ``energy criterion''  or cubic
for ``displacement'' criterion) of dimension $n$.
Here the optimal $\tr(X(\mathbf{p}_{\rm {opt}}, v_{\rm {opt}}, n))$ is defined with optimal position and corresponding viscosity
\begin{align*}
(\mathbf{p}_{\rm {opt}}, v_{\rm {opt}}) =
\argmin_{\mathbf{p}, v} \tr(X(\mathbf{p} , v ))\,.
\end{align*}

%
%As mentioned earlier, calculating the optimal position, even for "simple models" like the rod/string model, is computationally demanding, particularly when dealing with more than one damper.

%Therefore, we will proceed by employing a modal approximation for damping, specifically with $r>1$ equal dampers. For this approximated model, assuming diagonal damping $C_{\rm{ext}}$, the trace function $\tr(X(\mathbf{p}_{\rm {opt}}, v_{\rm {opt}}), n)$ becomes a linear function of dimension $n$. Unfortunately, in the general case \eqref{def_damping_string}, where the damping is not diagonal, the trace function can be quadratic or exhibit some other increasing polynomial in dimension $n$.

\subsection{Optimal position of one damper for the  discrete mass--spring model}
\label{Opt pos_single_damper2}

As previously highlighted, determining the optimal placement of several  dampers poses a significant challenge. As we will illustrate below, even in the case of the rod/string model, where explicit expressions for all undamped eigenfrequencies and eigenvectors are available, positioning multiple dampers remains an unresolved issue.

Thus, in the subsequent section, we will present our main results regarding the properties of the trace function of the solution to the Lyapunov equation in the case of one damper. These findings will enable us to identify the optimal position for a single damper in the discrete mass--spring model, applicable across all dimensions.

%{The proof below applies to the discrete mass-spring system from the previous section and can be easily generalized to a string or rod.}

The main tool will be a modified formula obtained by \textsc{Veseli\'{c}} in~\cite{VES90} in the form documented in~\cite{BRAB98} for the trace of the {Lyapunov} equation
\begin{equation} \label{LyapEq1Dnum}
A^T X + X A = -I_{2n},
\end{equation}
where $A$ is defined  as
\begin{equation}\label{Def_MatrixANum}
A =\begin{bmatrix} 0 & \Omega \\ -\Omega & - v \mathbf{c} \mathbf{c}^T   \end{bmatrix}\,,
\end{equation}
and  $v>0$  is the viscosity parameter which has to be optimized
and for a given position $\mathbf{c}=\begin{bmatrix} c_1 & c_2 & \ldots & c_n \end{bmatrix}^T$, $c_i \neq 0$.

The trace of the solution $X$ multiplied by diagonal matrix $Z_{\Delta}$ is given by
\begin{align}\label{Strucure_ofTraceXZ}
\tr( Z_{\Delta} X ) & = \frac{a}{v} +  b v \,,
\end{align}
where
\begin{align}\label{TraceofSolXZ}
 a & = \sum_{k=1}^n  \frac{2 z_{k}}{c_{k}^2} \; , \\
b & = \sum_{k=1}^n  \frac{z_{k} c_{k}^2}{2 \omega_k^2} +
 z_{k} \left( \sum_{j \neq k}^n
 \frac{3 \omega_k^2 c_{j}^2 + \omega_k^2 c_{k}^2 + \omega_j^2 c_j^2 + \omega_j^2 c_k^2}{ ( \omega_k^2 -\omega_j^2 )^2}
  + 2 \frac{\omega_k^2 }{ c_k^2} \left(\sum_{j \neq k}^n  \frac{ c_{j}^2}{ \omega_k^2 -\omega_j^2}\right)^2
  \right) \,, \label{TraceofSolXZ-b}
\end{align}
and $Z_{\Delta}$ is any diagonal matrix, i.e.
\begin{equation} \label{Def-Z_Delta}
Z_{\Delta} = \diag(z_1, \ldots, z_{n}, z_1, \ldots, z_{n})\,.
\end{equation}

From \eqref{Strucure_ofTraceXZ} we get optimal viscosity ``for free''
\begin{align}\label{OptimalVisc}
v_\textrm{opt} & = \sqrt{\frac{a}{b}} \,.
\end{align}

This means that the optimal trace for a damper in a fixed position is
\begin{align}\label{OptimalTraceXZ}
\tr(X_{\rm{opt}}) & = 2 \sqrt{ a b } \,.
\end{align}

In what follows we will show that the optimal trace \eqref{OptimalTraceXZ}
for position defined as in \eqref{DefposGeneral}, ($\mathbf{p}=k/n$) is asymptotically a linear function of the dimension $n$, that is
\begin{align*}
\tr(X_{\rm{opt}}) = \tr(X_{\rm{opt}}(\mathbf{p}, n)) \sim \mathcal{O}(n)\,.
\end{align*}
% where $X$ is a solution of the Lyapunov equation

More precisely, we will show that for the optimal trace \eqref{OptimalTraceXZ} at position $\mathbf{p}$, corresponding optimal coefficients $a$ and $b$ are linear functions of
the dimension $n$, i.e. we will show that for $n$ large enough ($n \rightarrow \infty$)
\begin{align*}
a(n) \sim \alpha_1 n + \alpha_0 \,,  b(n) \sim \beta_1 n + \beta_0 \,,
\end{align*}
for some functions $ \alpha_0,  \alpha_1$ and $\beta_0, \beta_1$ which does not depend on $n$.

\subsection{Optimal trace for the average total energy }
\label{Opt_trace_aver_Energy}

As we have mentioned above using expression $\eqref{OptimalTraceXZ}$
and the structure of $a$ and $b$ from \eqref{TraceofSolXZ} and \eqref{TraceofSolXZ-b}, respectively, we will prove that for rod or string model for $n$ large enough the optimal trace is  asymptotically
 a linear function of $n$. This means that for $n$ large enough the optimal trace divided by $n$ is a function that does not depend on $n$.

More precisely we will show that
\begin{align*} %\label{OptimalTraceXZ}
\lim\limits_{n \rightarrow \infty} \frac{a(n)}{n}
 = \alpha_1 \,, \quad
\lim\limits_{n \rightarrow \infty} \frac{b(n)}{n}
 = \beta_1 \,, \quad
\end{align*}
where $\alpha_1$ and $\beta_1$ does not depend on $n$, thus
\begin{align*} %\label{OptimalTraceXZ}
\lim\limits_{n \rightarrow \infty} \left(\tr(X(\mathbf{p}, n ) \right)
& =  \eta(\mathbf{p}, n)  \sim \eta_1 n + \eta_0 \,,
\end{align*}
where  $\eta_1$, and $\eta_0$ are some functions that are independent of $n$.

For the criterion of the average total energy, frequency-cut projection
is defined as in \eqref{DefZ2b}:
\begin{align*}
Z_1= Z_s \oplus Z_s\,, \qquad Z_s =\begin{bmatrix} I_{s} &  \\  & 0_{(n-s)}  \end{bmatrix}\,,
 \end{align*}
which means that all $z_k=1$ from \eqref{Def-Z_Delta},
for  $k=1, \ldots, s$, while $z_k=0$, for $k=s+1, \ldots, n$.

Recall that $ \omega_j$ and $c_j(\mathbf{p}, n)$ for string/rod model
are given as
\begin{align*}
 \omega_j(n) = 2 \sin{\frac{j \pi}{2 (n+1)}}\,,
  \end{align*}
and
\begin{align*}
c_j(\mathbf{p}, n) = \sqrt{\frac{2}{n+1}} \sin{\frac{j \mathbf{p} n \pi }{n+1}} \,.
\end{align*}
If we insert these into $a$ and $b$ from \eqref{TraceofSolXZ} and \eqref{TraceofSolXZ-b}, respectively,
we get
\begin{align}\label{Defabfor string}
 a(n) & = \sum_{k=1}^s  \frac{(n+1)}{\sin ^2\left(\frac{\pi  k n \mathbf{p}}{n+1}\right)} \, , \\
 b(n)& = b_1(n) + b_2(n) + b_3(n) \, , \\
b_1(n) & = \sum_{k=1}^s \frac{\csc ^2\left(\frac{\pi  k}{2(n+1)}\right) \sin ^2\left(\frac{\pi  k n \mathbf{p}}{n+1}\right)}{4 (n+1) }  \,, \\
b_2(n) & = \sum_{k=1}^s    \sum_{j \neq k}^n
 \frac{ \omega_k^2 ( 3 c_{j}^2 + c_{k}^2 ) + \omega_j^2 ( c_j^2 + c_k^2) }{ ( \omega_k^2 -\omega_j^2 )^2}
   \,, \\
 b_3(n) & = \sum_{k=1}^s \sum_{j \neq k}^n
   2 \frac{\omega_k^2 }{ c_k^2} \left(\sum_{j \neq k}^n  \frac{ c_{j}^2}{ \omega_k^2 -\omega_j^2}\right)^2
   \,,
\end{align}
 where, after additional simplification
 \begin{align}\label{b2andb3}
 b_2(n) & = \sum_{k=1}^s    \sum_{j \neq k}^n
 \frac{2 \sin ^2\left(\frac{\pi  j}{2 n+2}\right)
 \left(\sin^2\left(\frac{\pi  j n  \mathbf{p} }{n+1}\right)+\sin^2\left(\frac{\pi  k n  \mathbf{p} }{n+1}\right)\right)}{(n+1) \left(\cos \left(\frac{\pi  j}{n+1}\right)-\cos \left(\frac{\pi  k}{n+1}\right)\right)^2}
 \nonumber    \\
   &  + \frac{2 \sin^2\left(\frac{\pi  k}{2 n+2}\right) \left(3 \sin^2\left(\frac{\pi  j n \mathbf{p} }{n+1}\right)+2 \sin^2\left(\frac{\pi  k n  \mathbf{p} }{n+1}\right)\right)}{(n+1) \left(\cos \left(\frac{\pi  j}{n+1}\right)-\cos \left(\frac{\pi  k}{n+1}\right)\right)^2} \\
 b_3(n) & = \sum_{k=1}^s
   4 \cdot (n+1) \sin ^2\left(\frac{\pi  k}{2 n+2}\right) \csc ^2\left(\frac{\pi  k n  \mathbf{p} }{n+1}\right) \nonumber \\
   & \left(\sum_{j \neq k}^n \frac{\sin ^2\left(\frac{\pi  j n  \mathbf{p} }{n+1}\right)}{(n+1) \left(\cos \left(\frac{\pi  j}{n+1}\right)-\cos \left(\frac{\pi  k}{n+1}\right)\right)}  \right)^2
   \,, %
\end{align}
% Note
%\begin{align*}
% \frac{1}{v_\textrm{opt}} \tr(X(\mathbf{p},  n )
%& =  b(n) \,.
%\end{align*}

Thus, we proceed with the calculation  of the limits of
functions $a(n)/n$ and  $b(n)/n$.

For the first limit, note that
\begin{align*}
\frac{a(n)}{n} & =  \sum_{k=1}^s  \frac{(n+1)}{n \sin^2\left(\frac{\pi  k n \mathbf{p}}{n+1}\right)} \,,
\end{align*}
which implies that
\begin{align*} %\label{Limit_pat_der_a0}
\lim\limits_{n \rightarrow \infty} \frac{a(n)}{ n} & =
\sum_{k=1}^s \csc ^2(\pi  k \mathbf{p}) \,.
\end{align*}
This shows, that $a(n)$ has oblique asymptote with slope
\begin{align} \label{Limit_pat_der_a}
\alpha_1 = \sum_{k=1}^s \csc ^2(\pi  k \mathbf{p}) \,.
\end{align}

Further, note that
\begin{align*}
\frac{b(n)}{ n}& = \frac{ b_1(n)}{ n} + \frac{ b_2(n)}{ n} + \frac{ b_3(n)}{ n}\,.
\end{align*}
The ratio $\displaystyle{\frac{b_1(n)}{n}}$ is given by
\begin{align*}
\frac{b_1(n)}{n} & =  \sum_{k=1}^s \frac{\csc ^2\left(\frac{\pi  k}{2(n+1)}\right) \sin ^2\left(\frac{\pi  k n \mathbf{p}}{n+1}\right)}{4 n (n+1) } \,,
\end{align*}
thus, the limit of $\displaystyle{\frac{b_1(n)}{n}}$ is
\begin{align} \label{Limit_pat_der_b1}
\lim\limits_{n \rightarrow \infty} \frac{ b_1(n)}{ n} & =
\sum_{k=1}^s \frac{2 \sin ^2(\pi  k \mathbf{p})}{\pi ^2 k^2}\,.
\end{align}
Further,
\begin{align*}
\frac{b_2(n)}{n} & =  \sum_{k=1}^s    \sum_{j \neq k}^n
\Gamma(j,k,\mathbf{p}) \,,
\end{align*}
where
\begin{align*}
\Gamma(j,k,\mathbf{p}) & = % \sum_{k=1}^s    \sum_{j \neq k}^n
 \frac{2 \sin ^2\left(\frac{\pi  j}{2 n+2}\right)
 \left(\sin^2\left(\frac{\pi  j n  \mathbf{p} }{n+1}\right)+\sin^2\left(\frac{\pi  k n  \mathbf{p} }{n+1}\right)\right)}{n (n+1) \left(\cos \left(\frac{\pi  j}{n+1}\right)-\cos \left(\frac{\pi  k}{n+1}\right)\right)^2}
 \nonumber    \\
   &  + \frac{\sin^2\left(\frac{\pi  k}{2 n+2}\right) \left(3 \sin^2\left(\frac{\pi  j n \mathbf{p} }{n+1}\right)+2 \sin^2\left(\frac{\pi  k n  \mathbf{p} }{n+1}\right)\right)}{n (n+1) \left(\cos \left(\frac{\pi  j}{n+1}\right)-\cos \left(\frac{\pi  k}{n+1}\right)\right)^2}\,.
\end{align*}

Using simple properties of the limit of trigonometric functions and some trigonometric identities we can see that $\Gamma(j,k,\mathbf{p})$ has the
upper bound
% same rate of convergence as the sum
\begin{align*}
0 < \Gamma(j,k,\mathbf{p}) & \leq
 \frac{2 \left(2(j \pi)^2 + 4 (k \pi)^2  \right) }{n (n+1)(2 n+2)^2 \left(\sin \left(\frac{(j+k) \pi}{2(n+1)}\right) \sin \left(\frac{(j-k)\pi }{2(n+1)}\right)\right)^2}\,,
\end{align*}
which further implies
\begin{align*}
0 < \Gamma(j,k,\mathbf{p}) & \leq  \frac{4 (n+1) \left(j^2+2 k^2\right)}{\pi ^2 n (j-k)^2 (j+k)^2}\,.
\end{align*}
For a fixed $k \in \mathds{N}$, the sum
\begin{align} \label{defSk}
S(k)& = \frac{4}{\pi^2}\,\sum_{j \neq k}^\infty
 \frac{(n+1) \left(j^2+2 k^2\right)}{n (j-k)^2 (j+k)^2}\,
\end{align}
is convergent. Indeed note that  for fixed $k$
\begin{align*}
\sum_{j \neq k}^\infty
 \frac{(n+1) \left(j^2+2 k^2\right)}{n (j-k)^2 (j+k)^2}
 & =
\sum_{j \neq k}^\infty \frac{\left(j^2+2 k^2\right)}{(j^2-k^2)^2} \\
& = \sum_{j \neq k}^\infty \frac{1}{(j^2-k^2)} +
3 \sum_{j \neq k}^\infty \frac{k^2}{(j^2-k^2)^2}  \\
& = \sum_{j = 1}^{k-1} \frac{1}{(j^2-k^2)} +
\sum_{j = k+1}^\infty \frac{1}{(j^2-k^2)} +
3 \sum_{j = 1}^{k-1} \frac{k^2}{(j^2-k^2)^2}    +
3 \sum_{j = k+1}^\infty \frac{k^2}{(j^2-k^2)^2}    \,.
\end{align*}
Since, for $j>k$ we have
\begin{align*}
\frac{k^2}{(j^2-k^2)^2} \leq \frac{k^2}{j^2-k^2} \,,
\end{align*}
using
\begin{align*}
\sum_{j = k+1}^{n} \frac{1}{ j^2 - k^2} & = \frac{1}{2 k} H_{2k}  \,,%  -
\end{align*}
where $H_{2k}$ are generalized harmonic numbers ( see \cite[Theorem 3.4]{BBCLGM2007})
% (for more details see Sondow, Jonathan; Weisstein, Eric W. "Riemann Zeta Function (eq. 52)". MathWorld—A Wolfram Web Resource),
we can conclude that the sum $S(k)$ from \eqref{defSk} is convergent.

For example (for $n\rightarrow \infty$)
\begin{align*}
S(1)=0.468064 \,, S(2)=0.867021 \,, S(3)= 0.940901 \,, S(4)=0.96676 \,, \ldots
\end{align*}
\begin{align*}
S(10)=0.994687 \,, \ldots, S(100)= 0.999953 \,, \ldots, S(n)= 1\,.
\end{align*}
Thus, the limit of $\displaystyle{\frac{b_2(n)}{n}}$ is bounded by
\begin{align} \label{Limit_pat_der_b2}
\lim\limits_{n \rightarrow \infty} \frac{b_2(n)}{n} & \leq
 s \cdot \lim\limits_{n \rightarrow \infty}   \sum_{j \neq k}^n
\frac{4 (n+1) \left(j^2+2 k^2\right)}{\pi ^2 n (j-k)^2 (j+k)^2} \leq s\,.
\end{align}
Note that the above limit depends only on $s$  and it is independent on $\mathbf{p}$
(position of the damper).

For the last limit, consider
\begin{align*}
\frac{b_3(n)}{n} & = \sum_{k=1}^s
   \frac{4}{n(n+1)} \sin ^2\left(\frac{\pi  k}{2 n+2}\right) \csc ^2\left(\frac{\pi  k n  \mathbf{p} }{n+1}\right) \nonumber \\
   & \left(\sum_{j \neq k}^n \frac{\sin ^2\left(\frac{\pi  j n  \mathbf{p} }{n+1}\right)}{\left(\cos \left(\frac{\pi  j}{n+1}\right)-\cos \left(\frac{\pi  k}{n+1}\right)\right)}  \right)^2    \,.
\end{align*}
Note that
\begin{align*}
\lim\limits_{n \rightarrow \infty} \frac{b_3(n)}{n} & = \lim\limits_{n \rightarrow \infty} \sum_{k=1}^s
   \frac{\left(\pi ^2 k^2\right) \csc ^2\left(\frac{\pi  k n \mathbf{p}}{n+1}\right)}{n (n+1)}
    \left(\sum_{j \neq k}^n \frac{2 \sin^2\left(\frac{\pi  j n \mathbf{p}}{n+1}\right)}{\pi^2 \left(k^2-j^2\right)}  \right)^2   \\
   & = \sum_{k=1}^s \csc^2(\pi  k \mathbf{p}) \left(\sum _{j \neq k}^\infty
   \frac{2 k \sin^2(\pi  j \mathbf{p})}{\pi  \left(k^2-j^2\right)}\right)^2 \,.
   %& \sim \sum_{k=1}^s \left(\sum_{j \neq k}^\infty \frac{n+1}{n \left(j^2-k^2\right)}  -  \sum_{j \neq k}^\infty \frac{(n+1) \cos(2 \pi  j \mathbf{p}) }{n \left(j^2-k^2\right)}\right)^2
   \end{align*}
Note, for  fixed $k$
\begin{align*}
  \left|
   \frac{\sin^2(\pi  j \mathbf{p})}{\left(k^2-j^2\right)} \right| & \leq
   \left| \frac{1}{ j^2 - k^2}  \right| \,,
   \end{align*}
which using
\begin{align*}
     \sum_{j = k+1}^\infty \frac{1}{ j^2 - k^2}  & = \frac{1}{2 k} H_{2k}\,,
   \end{align*}
where $H_{2k}$ are generalized harmonic number,
implies that the sums
\begin{align*}
\sum_{k=1}^s \csc^2(\pi  k \mathbf{p}) \left(\sum _{j=1}^\infty
\frac{2 k \sin^2(\pi  j \mathbf{p})}{\pi  \left(k^2-j^2\right)}\right)^2\,,
    \end{align*}
is convergent. We have
\begin{align} \label{Limit_pat_der_b3}
\lim\limits_{n \rightarrow \infty} \frac{b_3(n)}{n} & \leq
 \sum_{k=1}^s
 \csc^2(\pi  k \mathbf{p}) % \lim\limits_{n \rightarrow \infty}
 T(k,\mathbf{p})\,,
\end{align}
where
\begin{align} \label{DefTodn}
T(k,\mathbf{p})&=\sum_{j =1}^k \frac{1}{k^2-j^2} + \frac{1}{2 k} H_{2k} \,.
\end{align}
Note that the limit in \eqref{Limit_pat_der_b3} depends only on $s$ and $\mathbf{p}$.

We have showed that sums from \eqref{Limit_pat_der_b1}, \eqref{Limit_pat_der_b2} and \eqref{Limit_pat_der_b3} converge, respectively, which implies
\begin{align} \label{LimitBovern}
\lim\limits_{n \rightarrow \infty} \frac{b(n)}{n}
 = \beta_1(s,\mathbf{p}) \,, \quad
\end{align}
and from \eqref{Limit_pat_der_a} follows
\begin{align} \label{LimitAovern}
\lim\limits_{n \rightarrow \infty} \frac{a(n)}{n}
 = \alpha_1(s,\mathbf{p}) \,. \quad
\end{align}

To summarize, if one is interested in the calculation of the
the best damping, that is the best position and corresponding optimal viscosity using the average total energy criterion,  is equivalent to the  minimization
\begin{align*}
 \tr(X(\mathbf{p} , v , s)) \rightarrow \min\,,
\end{align*}
where $X(\mathbf{p} , v , s)$ is the solution of the {Lyapunov} equation
\begin{equation}  \label{LyapEqEnergyCrit}
A X(\mathbf{p} , v , s) + X(\mathbf{p} , v , s) A^T = -Z,
\end{equation}
where $A$ is defined  in \eqref{Def_MatrixANum}
%\begin{equation*} %\label{Def_MatrixANum}
%A =\begin{bmatrix} 0 & \Omega \\ -\Omega & - v \mathbf{c} \mathbf{c}^T   \end{bmatrix}\,,
%\end{equation*}
and  $v>0$  is the viscosity parameter which has to be optimized
for a damper $\mathbf{c}= c(\mathbf{p})$ in given position
$\mathbf{p} = k/n$.

If one tries to calm down the first $s$ undamped eigenfrequencies $\omega_1 < \ldots < \omega_s$, the matrix $Z$ is defined as
\begin{equation*}
Z = Z_s \oplus Z_s \,, Z_s = \begin{bmatrix} I_s &  0\\ 0 & 0_{n-s} \end{bmatrix}\,.
\end{equation*}
Thus, let
\begin{align} \label{OptViscEnergyCrit}
v_{\rm {opt}} = \argmin_{v} \tr(X(\mathbf{p} , v , s))\,,
\end{align}
be the optimal viscosity obtained for the position $\mathbf{p}$ and given $s$.

Based on the preceding calculation, we can state the following theorem.

\begin{theorem} \label{Theorem1}
Let $X(\mathbf{p} , v_{\rm {opt}} , s)$ be the solution of the Lyapunov equation \eqref{LyapEqEnergyCrit}, with optimal viscosity
$v_{\rm {opt}}$ as in \eqref{OptViscEnergyCrit}.  Then it holds
\begin{align} \label{LimitofOptTraceEnergyCrit}
\lim\limits_{n \rightarrow \infty}
\frac{\tr(X(\mathbf{p}, v_{\rm {opt}},s))}{n} & =  \eta_1(\mathbf{p}, s) \,,
\end{align}
where  $\eta_1(\mathbf{p}, s)$, is a function that is independent of $n$.
\end{theorem}
\begin{proof} Using \eqref{OptimalTraceXZ}, one can see that
\begin{align*}
\frac{\tr(X(\mathbf{p}, v_{\rm {opt}},s))}{n}  & = 2 \sqrt{ \frac{a(n)}{n} \frac{b(n)}{n} } \,.
\end{align*}
Now, from \eqref{LimitAovern}  and \eqref{LimitBovern} follows
\begin{align*}
\lim\limits_{n \rightarrow \infty} \frac{\tr(X(\mathbf{p}, v_{\rm {opt}},s))}{n}  & = 2 \sqrt{ \alpha_1(s,\mathbf{p}) \beta_1(s,\mathbf{p}) } \doteq \eta_1(\mathbf{p}, s)\,,
\end{align*}
which completes the proof.
\end{proof}

The main benefit from the result of Theorem \ref{Theorem1} is
that, for a given set of dominant eigenfrequencies which we try to calm,
that is for a given matrix $Z$, one can calculate the best position
with the corresponding optimal viscosity for some modest dimension $n$,
which then holds for general (large) $n$.  Note, that the optimal position $\mathbf{p}$ will change if we change the number of dominant eigenfrequencies $s$ which we try to calm.

All this will be illustrated in the section with numerical examples.

\subsection{Optimal trace for the total average displacement }
\label{Opt_trace_aver_Displac}

To minimize the total average displacement within the rod or string model, it is necessary to minimize the trace of the solution to the Lyapunov equation. This involves using the same system matrix but with a distinct right-hand side.

Recall, from \eqref{defhatK} and \eqref{Zfordisplacmt} follows that the right-hand side of the Lyapunov equation contains
\begin{align*}
   \widehat{K}^{-1} =  \Omega^{-1} \Phi^{T} \Phi \Omega^{-1} = \Omega^{-2}\,.
\end{align*}
Here we have used that $\Phi$ is an orthogonal matrix.

Thus, the projection matrix $Z_{\Delta}$ from \eqref{Strucure_ofTraceXZ} and
\eqref{Zfordisplacmt} is diagonal given as
\begin{align*}
Z_{\Delta} = \diag(\Omega_s^{-2}, 0_{n-s}) \oplus 0_n \,,
\end{align*}
where
\begin{align*}
\Omega_s= \diag(\omega_1, \ldots, \omega_s) \,.
\end{align*}
This configuration enables us (once again) a usage of $a$ and $b$  from  \eqref{TraceofSolXZ} and \eqref{OptimalTraceXZ} for determining the optimal trace.

Note, if we are interested in calming down $s$ dominant eigenfrequencies
$\omega_1, \ldots, \omega_s$, the non-zero diagonal entries of the matrix
are given as
\begin{align*}
z_k = \frac{1}{\omega_k^2}\,, \qquad k=1, \ldots, s\,.
\end{align*}
Since, for the string/rod model $\omega_k$ is given as
\begin{align*}
\omega_k = 2 \sin\frac{\pi k}{2 n+2} \,,
\end{align*}
we see that $\omega_k^2$ has the same rate of
convergence as $\displaystyle{\frac{1}{n^2}}$ i.e.
\begin{align} \label{limitomega2n2}
\lim\limits_{n \rightarrow \infty} n^2 \omega_k^2 = k^2 \pi^2 \,.
\end{align}

Thus, in the case of the minimization of the total average displacement,
for $n$ large enough the optimal trace function divided
by $n^3$ is a function which does not depend on $n$.

More precisely, from \eqref{OptimalTraceXZ} follows
\begin{align} \label{OptimalTraceXZDispl}
\tr(X(\mathbf{p}, v_{\rm {opt}},s)) & = 2 \sqrt{ a_K(n) b_K(n) } \,,
\end{align}
where
\begin{align} \label{TraceDisplacement} %\label{Defabfor string}
 a_K(n) & = \sum_{k=1}^s  \frac{2}{\omega_k^2 c_{k}^2} \; , \\
b_K(n) & = \sum_{k=1}^s  \frac{c_{k}^2}{2 \omega_k^4}   \\
 & + \frac{1}{\omega_k^2} \left( \sum_{j \neq k}^n
 \frac{2 \omega_k^2 c_{j}^2 + \omega_k^2 c_{k}^2
  %+\omega_j^2 c_j^2 + \omega_j^2 c_k^2
  }{ ( \omega_k^2 -\omega_j^2 )^2}
  + \frac{\omega_k^2 }{ c_k^2} \left(\sum_{j \neq k}^n  \frac{ c_{j}^2}{ \omega_k^2 -\omega_j^2}\right)^2
  \right) \,. \nonumber
\end{align}
Now, using \eqref{Limit_pat_der_a} and \eqref{limitomega2n2} it is easy to show
\begin{align} \label{LimitaK}
\lim\limits_{n \rightarrow \infty} \frac{a_K(n)}{n^3} & =
\sum_{k=1}^s   \frac{1}{k^2 \pi^2}  \csc ^2(\pi  k \mathbf{p}) \,.
\end{align}
Further let $b_{K1}$, $b_{K2}$ and $b_{K3}$ correspond to
$b_{1}$, $b_{2}$ and $b_{3}$, from \eqref{Defabfor string},
such that $b_{Ki}$ is obtained from $b_{i}$ by multiplying its elements
from the first sum with $1/\omega_k^2$, $i=1, 2, 3$.

Now again using \eqref{limitomega2n2} and \eqref{Limit_pat_der_b1} it is easy to get
\begin{align} \label{Limitb1K}
\lim\limits_{n \rightarrow \infty} \frac{ b_{K1}(n)}{n^3} & =
\sum_{k=1}^s \frac{2 \sin ^2(\pi  k \mathbf{p})}{\pi^4 k^4}\,.
\end{align}
Similarly, using \eqref{Limit_pat_der_a} and \eqref{Limit_pat_der_b2}
we get
\begin{align} \label{Limitb2K}
\lim\limits_{n \rightarrow \infty} \frac{b_{K2}(n)}{n^3} & \leq
\sum_{k=1}^s \frac{1}{k^2 \pi^2} \,,
\end{align}
and using \eqref{limitomega2n2} and \eqref{Limit_pat_der_b3},
we have
\begin{align} \label{Limitb3K}
\lim\limits_{n \rightarrow \infty} \frac{b_{K3}(n)}{n^3} & \leq
 \sum_{k=1}^s  \frac{1}{k^2 \pi^2} \csc^2(\pi  k \mathbf{p}) T(k,\mathbf{p})\,,
\end{align}
where $T(k,\mathbf{p})$ is defined as in \eqref{DefTodn}, i.e.
\begin{align*}
T(k,\mathbf{p})&=\sum_{j =1}^k \frac{1}{k^2-j^2} + \frac{1}{2 k} H_{2k} \,.
\end{align*}

Now similarly as in the previous section from \eqref{LimitaK},
and  \eqref{Limitb1K}, \eqref{Limitb2K}  and \eqref{Limitb3K} we
know that limits exist, thus we can write
\begin{align} \label{LimitAKovern3}
\lim\limits_{n \rightarrow \infty} \frac{a(n)}{n^3}
 = \gamma_1(s,\mathbf{p}) \,,
\end{align}
and
\begin{align} \label{LimitBKovern3}
\lim\limits_{n \rightarrow \infty} \frac{b_K(n)}{n^3}
 = \gamma_2(s,\mathbf{p}) \,. \quad
\end{align}

To summarize, for the total average displacement, if one tries to calm down the first $s$ undamped eigenfrequencies $\omega_1 < \ldots < \omega_s$, the matrix $Z$ is defined as
\begin{align} \label{ZDicplac}
Z = \diag(\Omega_s^{-2}, 0_{n-s}) \oplus 0_n \,,
\end{align}
where
\begin{align*}
\Omega_s= \diag(\omega_1, \ldots, \omega_s) \,.
\end{align*}
Thus, let
\begin{align} \label{OptViscEnergyCritDis}
v_{\rm {opt}} = \argmin_{v} \tr(X(\mathbf{p} , v , s))\,,
\end{align}
be the optimal viscosity obtained for the position $\mathbf{p}$ and given $s$ using the total average displacement criterion.

We can state the following theorem.

\begin{theorem} \label{Theorem2}
Let $X(\mathbf{p} , v_{\rm {opt}} , s)$ be the solution of the Lyapunov equation \eqref{LyapEqEnergyCrit}, with $Z$ from \eqref{ZDicplac} and optimal viscosity $v_{\rm {opt}}$ as in \eqref{OptViscEnergyCrit}.  Then it holds
\begin{align} \label{LimitofOptTraceDisplacCrit}
\lim\limits_{n \rightarrow \infty}
\frac{\tr(X(\mathbf{p}, v_{\rm {opt}},s))}{n^3} & =  \xi(\mathbf{p}, s) \,,
\end{align}
where  $\xi(\mathbf{p}, s)$, is a function that is independent of $n$.
\end{theorem}
\begin{proof} Using \eqref{OptimalTraceXZDispl}, one can see that
\begin{align*}
\frac{\tr(X(\mathbf{p}, v_{\rm {opt}},s))}{n^3}  & = 2 \sqrt{ \frac{a_K(n)}{n^3} \frac{b_K(n)}{n^3} } \,.
\end{align*}
Now, from \eqref{LimitAKovern3}  and \eqref{LimitBKovern3} follows
\begin{align*}
\lim\limits_{n \rightarrow \infty} \frac{\tr(X(\mathbf{p}, v_{\rm {opt}},s))}{n^3}  & = 2 \sqrt{ \gamma_1(s,\mathbf{p}) \gamma_2(s,\mathbf{p}) } \doteq \xi(\mathbf{p}, s)\,,
\end{align*}
which completes the proof.
\end{proof}

%
%
%O L D T E X T
%
%
%This means that for a given position $\mathbf{p} \cdot n$ it holds
%\begin{align*} %\label{OptimalTraceXZ}
%\lim\limits_{n \rightarrow \infty} \frac{a(n)}{n^3}
% = \alpha_K \,, \quad
%\lim\limits_{n \rightarrow \infty} \frac{b(n)}{n^3}
% = \beta_K \,, \quad
%\end{align*}
%This further means that
%\begin{align*} %\label{OptimalTraceXZ}
%\lim\limits_{n \rightarrow \infty} \left(\tr(X(\mathbf{p}, n ) \right)
%& =  \eta_K(\mathbf{p}, n)  \sim  \sum_{k=0}^3 (\eta_K)_{k} n^{k} \,.
%\end{align*}

Typically, the optimal position achieved through minimizing the total average energy may not align with the optimal position attained through minimizing the total average displacement. However, there are instances where they can be closely proximate or even identical. This will be illustrated in the next section with numerical examples.

The main benefit of Theorem \ref{Theorem2} is similar to the benefit of
Theorem \ref{Theorem1}, that is for a given set of dominant eigenfrequencies which we try to calm,
that is for a given matrix $Z$, one can calculate the best position
with the corresponding optimal viscosity for some modest dimension $n$,
which then holds for general (large) $n$.  Note, that the optimal position $\mathbf{p}$ will change if we change the number of dominant eigenfrequencies $s$ which we try to calm.

\begin{remark}
It is important to note that in the case of string or rod oscillations,
the eigenfrequencies and corresponding eigenvectors do not depend on the dimension $n$.

For the string, which is explicitly diagonalized with the undamped frequencies
\[
\omega_k = k \pi
\]
the orthonormal eigenvectors are given by
\[
u_k(x) = \frac{1}{\sqrt{2}} \sin(k \pi x).
\]

For the rod, the eigenvectors are also
\[
u_k(x) = \frac{1}{\sqrt{2}} \sin(k \pi x),
\]
but the eigenfrequencies are
\[
\omega_k = k \pi \sqrt{k^2 \pi^2 a_0 + k_0}.
%\omega_k = k \pi \sqrt{k_0 + a_0 k^2}.
\]

Since the summations in equations~\eqref{TraceofSolXZ} and~\eqref{TraceofSolXZ-b} run from $k = 1$ to $k = s$ for a fixed $s$,
the quantities $a$ and $b$ remain equal for any $n > s$.
This is due to the fact that the first $s$ eigenpairs are independent of the total dimension $n$.
\end{remark}

\begin{remark}

Regarding the relationship between the discretization size $s$ and the accuracy of the numerical solution, a classical engineering example presented in Reddy \cite{reddy2005} demonstrates that for linear finite elements applied to the one-dimensional wave equation with fixed ends, the $L^2$-norm error of the solution converges with order $\mathcal{O}(h^2)$, where $h = 1/n$ is the uniform mesh size. This result is based on the comparison between the analytical eigenmode solution and the finite element approximation using a consistent mass matrix.

Therefore, if we aim for an accuracy of approximately $\varepsilon = 10^{-8}$, it is sufficient to choose the number of elements such that $n > 10000$, under the assumption that $s < n$. This ensures that the number of retained modes remains below the spatial resolution limit imposed by the mesh, maintaining both stability and accuracy of the FEM approximation.
\end{remark}

\section{Numerical illustration of the main results \label{sec_numerical_example}}

As an illustration of the presented results, we will investigate the mechanical system from section \ref{sec_optimal_positions} where all masses and stiffness are equal to $1$.

\begin{example} % \label{IllustrEx1}
Consider the mechanical system from Figure~\ref{fig1} with all $m_i=k_i=1 = k_{n+1}$, $i=1, \ldots, n$. This means that we  {are} looking for the best position for the damper with optimal viscosity $v$ for the mechanical system described by equation~\eqref{MnSys} where
\begin{align*}
 M=  I_{n} \,,  \qquad  K = \begin{bmatrix}
  2 & -1 &  &  &  \\
  -1    & 2& -1 &  &  \\
          & \ddots   & \ddots & \ddots &  \\
          &      & -1 & 2 & -1   \\
          &      &  & -1 & 2 \end{bmatrix} \,.
\end{align*}
\end{example}

%For internal damping we set $0.02 \cdot \Omega$, so that
%\begin{align*}
% A_0= \begin{bmatrix}  0  & \Omega \\
%  -\Omega & -0.02 \cdot \Omega  \end{bmatrix} \,.
%\end{align*}
We consider the  optimal placing of just one damper, for both criteria.
In the first part of this section, we will calculate
the optimal position for dimensions
\[
n\in\{2000, 3000, \ldots, 10000 \}\,.
\]

\textbf{Criterion: the total average energy}

For the case when one tries to calm down the first $s=100$ dominant frequencies,
($\omega_1, \omega_2, \ldots, \omega_s$), the optimal position for
the total average energy is:
\[
\mathbf{p} = 0.495\,.
\]
More precisely, table \ref{Table1} shows all optimal positions for all
dimensions \\ $n\in\{2000, 3000, \ldots, 10000 \}$.
\begin{table}[!h]
\begin{tabular}{|c||c|c|c|c|c|c|c|c|c|}
  \hline
  dim ($n$) & 2000 & 3000 & 4000 & 5000 & 6000 & 7000 & 8000 & 9000 & 10000 \\ \hline
  opt. pos. & 991  & 1487 & 1982 & 2477 & 2973 & 3468 &  3963 &   4458 &  4950 \\
  \hline
\end{tabular}
\caption{Optimal positions for $\omega_1 < \ldots <\omega_{100}$.}
\label{Table1}
\end{table}

On the other hand, for the case when one tries to calm down $s=100$ frequencies starting from $100$, that is $\omega_{101}, \omega_{102}, \ldots, \omega_{200}$, the matrix $Z$ will be defined as
\begin{equation*}
Z= Z_{(100, s)} \oplus Z_{(100, s)}\, \qquad
Z_{(100, s)} =\begin{bmatrix}0_{100} &  &  \\
 & I_{s} &  \\  & &  0_{(n-s-100)}  \end{bmatrix}\,.
\end{equation*}

The optimal position is:
\[
\mathbf{p} = 0.0032\,.
\]
Table \ref{Table2} below shows all optimal positions for all
dimensions \\ $n\in\{2000, 3000, \ldots, 10000 \}$.

\begin{table}[!h]
\begin{tabular}{|c||c|c|c|c|c|c|c|c|c|}
  \hline
  dim ($n$) & 2000 & 3000 & 4000 & 5000 & 6000 & 7000 & 8000 & 9000 & 10000 \\ \hline
  opt. pos. &   6    &  10    & 13    &16    &19    &23     &  26   &  28  & 32    \\
  \hline
\end{tabular}
\caption{Optimal positions for $\omega_{101} < \ldots <\omega_{200}$.}
\label{Table2}
\end{table}

We note that for the first case, the optimal position is close to the center of the ``structure'' (or chain) while in the second case, the optimal position is close to one of the edges.

A similar conclusion (or property) holds if we consider the first $s$ dominant frequencies for $s=20$, that is
$\omega_1, \omega_2, \ldots, \omega_{20}$. The best position is:
\[
\mathbf{p} = 0.4785\,, % \qquad \mathbf{p}=0.4785\cdot n \,,
\]
while for the eigenfrequencies starting from $100$-th ($\omega_{101}, \omega_{102}, \ldots, \omega_{120}$) the best position is
\[
\mathbf{p} = 0.0045\,. % \qquad \mathbf{p}=0.045 \cdot n \,,
\]

\textbf{Criterion: the total average displacement}

For the case when one tries to calm down the first $s=100$ dominant frequencies,
($\omega_1, \omega_2, \ldots, \omega_s$), using
the total average displacement, the optimal position is:
\[
\mathbf{p} = 0.419 \,, % \qquad \mathbf{p}=0.419 \cdot n \,,
\]
or as one can see from table \ref{Table3} for all $n\in\{2000, 3000, \ldots, 10000 \}$
\begin{table}[!h]
\begin{tabular}{|c||c|c|c|c|c|c|c|c|c|}
  \hline
  dim ($n$) & 2000 & 3000 & 4000 & 5000 & 6000 & 7000 & 8000 & 9000 & 10000 \\ \hline
  opt. pos. &   838    &  1260    &  1658   & 1940   &   2481 &  2893        & 3359    & 3779   &  4199   \\
  \hline
\end{tabular}
\caption{Optimal positions for $\omega_1 < \ldots <\omega_{100}$.}
\label{Table3}
\end{table}

On the other hand, for the case when one tries to calm down $s=100$ frequencies starting from $100$, that is ($\omega_{101}, \omega_{102}, \ldots, \omega_{200}$), the optimal position is:
\[
\mathbf{p} = 0.0034\,. %,  \qquad  \mathbf{p}=0.0034\cdot n
\]
 Table \ref{Table4} contains all optimal position for all $n\in\{2000, 3000, \ldots, 10000 \}$.

\begin{table}[!h]
\begin{tabular}{|c||c|c|c|c|c|c|c|c|c|}
  \hline
  dim ($n$) & 2000 & 3000 & 4000 & 5000 & 6000 & 7000 & 8000 & 9000 & 10000 \\ \hline
  opt. pos. & 7 & 10 &  14 & 18 & 21 & 24 & 28 &  32 &  34 \\
  \hline
\end{tabular}
\caption{Optimal positions for $\omega_{101} < \ldots <\omega_{200}$.}
\label{Table4}
\end{table}

Note that, here similarly as in the ``energy criterion'' the optimal position, for the first case, is closer to the center of the ``structure'' (or chain)  although positions do not coincide, while for the second  case, optimal position is close to one of the edges.

Similarly as in the  ``energy criterion'', if we consider the first $s=20$ dominant eigenfrequencies ($\omega_1, \omega_2, \ldots, \omega_{20}$), the optimal position is:
\[
\mathbf{p} = 0.3813\,, % \qquad \mathbf{p}=0.3813 \cdot n \,,
\]
while for the eigenfrequencies starting from $100$-th ($\omega_{101}, \omega_{102}, \ldots, \omega_{120}$) optimal position is
\[
\mathbf{p} = 0.0045\,. % \qquad \mathbf{p}=0.045 \cdot n \,.
\]

We can see that for a certain set of dominant eigenfrequencies, each criterion gives a different optimal position. Thus, as we have mentioned in one of the previous sections, the investigation of properties of both criterion and mutual comparison will be one of our future studies.

Just as the illustration for mutual comparison between both criteria, let us consider one more example where we will change the range of the dominant
eigenfrequencies $\omega_{i+1}, \omega_{i+2}, \ldots, \omega_{i+s}$.

Thus, the tables below contain the optimal position for both criteria
for a different set of the dominant eigenfrequencies.

\begin{table}[!h]
\begin{tabular}{|c||c|c|c|c|c|c|}
  \hline
  $(i+1)-(i+s)$ & $1-20$ & $1-50$ & $1-100$ & $101-120$ & $101-150$ & $201-220$ \\ \hline
  $\mathbf{p}$; \,\, Av. En. & 0.479 & 0.491 & 0.495  & 0.0045 & 0.004 &  0.0024     \\
  \hline
  $\mathbf{p}$; \,\, Av. Dis. & 0.381 & 0.382  & 0.42  & 0.0045 & 0.004 & 0.0024  \\
  \hline
\end{tabular}
\caption{Comparison between ``energy'' and ``displacement'' criterion.}
\label{Table5}
\end{table}

\quad

or

\quad

\begin{table}[!h]
\begin{tabular}{|c||c|c|c|c|c|c|}
  \hline
  $(i+1)-(i+s)$ & $1-5$ & $1-10$ & $11-15$ & $11-20$ & $31-35$ & $31-40$ \\ \hline
  $\mathbf{p}$; \,\, Av. En. & 0.435  & 0.459  &  0.0385  & 0.0317  & 0.045  & 0.014      \\
  \hline
  $\mathbf{p}$; \,\, Av. Dis. & 0.424  & 0.417   &  0.039  & 0.035   & 0.0455  &  0.0143 \\
  \hline
\end{tabular}
\caption{Comparison between ``energy'' and ``displacement'' criterion.}
\label{Table6}
\end{table}

From  both Tables \ref{Table5} and \ref{Table6}, one can see that in most of the cases the optimal positions coincide or they are close. The major difference is for the cases when one considers the first part of the spectrum.

\section{Conclusion and outlook}
\label{sec_conclusion}

The paper's main contribution is the study of two distinct optimization criteria for damping optimization in a multi-body oscillator system with arbitrary degrees of freedom ($n$),
which corresponds with the model of string/rod free vibrations. As the first result, we have shown that both criteria are equivalent with the trace minimization of the solution of the
Lyapunov equation with different right-hand sides. The second result we prove that the minimal trace for each criterion can be expressed as a simple (linear or cubic) function of
dimension $n$. In other words, the optimal position solely depends on the number of dominant eigenfrequencies and the optimal viscosity and does not depend on dimension $n$.
This highlights a simplified approach to damping optimization in such systems.

\textbf{Conjectures and future work:}

\textbf{1.}  Numerical experiments show the similar results
about the optimal positions of several dampers hold as well. Namely, the optimal positions of $r\geq2 $ dampers depend only on the number of dominant eigenfrequencies $s$  and their location $i$ ($\omega_{i+1} < \ldots < \omega_{i+s}$)  and does not depend on dimension $n$. It is true for both criteria.

We need to prove this, and we intend to tackle it in the near future.

\textbf{2.}  Study, the system of ODEs of the more general structure
like the vibration {chain} of masses and springs shown {on} Figure \ref{fig1} with general $m_i$ and $k_i$.

\textbf{3.}  Some general theoretical results about the stability of infinite dimensional systems with bounded spectrum. For example a deeper understanding of the spectral properties  with one or two or more dimensional dimensional damping.

\textbf{4.} Further study of the properties of both optimization criteria:
i) \textbf{The total average energy over all possible initial data}
and  ii) \textbf{The total average displacement over all possible initial data.} Especially a comparison between them in the sense of the quality of the solution and the complexity of the calculation.

%\begin{acknowledgements}
\textbf{Acknowledgements:}
The first author was supported in part by the Croatian Science Foundation under the project Conduction, IP-2022-10-5191.
Both authors appreciate the reviewers' valuable feedback and insightful comments, which have improved the quality of this manuscript.

%\end{acknowledgements}

\bibliographystyle{plain}
\bibliography{bibliogr_opt_ab}

\end{document}